\documentclass{article}%
\usepackage{amsmath}
\usepackage{amsfonts}
\usepackage{amssymb}
\usepackage{graphicx}%
\usepackage[pagewise]{lineno}

\usepackage{cite}

\newtheorem{theorem}{Theorem}[section]

\newtheorem{corollary}[theorem]{Corollary}

\newtheorem{definition}[theorem]{Definition}

\newtheorem{lemma}[theorem]{Lemma}

\newtheorem{proposition}[theorem]{Proposition}
\newtheorem{remark}[theorem]{Remark}

\newenvironment{proof}[1][Proof]{\noindent\textbf{#1.} }{\ \rule{0.5em}{0.5em}}
\numberwithin{equation}{section}

\begin{document}

\title{Generalized Hannay-Berry Connections on Foliated Manifolds and Applications}

\author{Misael Avenda\~no-Camacho, Isaac Hasse-Armengol \&
Yury Vorobev \\
Department of Mathematics, University of Sonora \\
Blvd. Luis Encinas y Rosales, \\
Col. Centro, C.P. 83000\\
Hermosillo, Sonora, M\'exico}

\maketitle

\begin{abstract}

\end{abstract}

\section{Introduction}

In this paper we discuss some aspects of the averaging method for Poisson
connections on foliated manifolds with symmetry generalizing the previous
results on the Hannay-Berry connections on fibrations due to \cite{Mn-88,MaMoRa-90} which play an important role in the normal form theory for
Hamiltonian systems of adiabatic type (see, for example, \cite{MYu-17}). One
of our main motivations is related to the further development and reviewing
of the averaging procedure for Dirac structures with singular presymplectic
foliations \cite{VaVo-14}.

Our starting point is a Poisson foliation $(M,\mathcal{F},P)$ consisting of a
regular foliation $\mathcal{F}$ on a manifold $M$ and a vertical Poisson
tensor $P$ on $M$ characterizing by the condition: each symplectic leaf of $P$
belongs to a leaf of $\mathcal{F}$. We are interested in the set
$\operatorname{Conn}_{H}(M,\mathcal{F},P)$ of Ehresmann-Poisson connections
$\gamma$ on $(M,\mathcal{F},P)$ satisfying the following condition: the
curvature of $\gamma$ is Hamiltonian, that is, the curvature form takes values
in Hamiltonian vector fields of $P$. Thus, we can assign to each connection  $\gamma\in \operatorname{Conn}_H(M,\mathcal{F},P) $ a horizontal $2$-form called the Hamiltonian form of the curvature. This form is uniquely determined modulo Casimir-valued horizontal 2-form. A Poisson connection $\gamma\in \operatorname{Conn}_H(M,\mathcal{F},P) $ is said to be admissible if  one can choose its Hamiltonian 2-form of the curvature to be " horizontally closed". Such class of Poisson connections with
Hamiltonian curvature naturally arises in the context of the coupling method
for Poisson and Dirac structures on fibered and foliated manifolds
\cite{Vo-01}, \cite{Va-06},. In particular, it is well known that each  coupling Dirac structure induces an admissible Poisson connection, \cite{Cu-90,Va-06,Wa-08}. Conversely, the
coupling procedure actually gives the conditions under which the vertical
Poisson tensor $P$ can be extended to a special Dirac structure via a given
connection $\gamma\in\operatorname{Conn}_{H}(M,\mathcal{F},P)$. In general, the set
$\operatorname{Conn}_{H}(M,\mathcal{F},P)$ can be empty. In the case of
 fibrations, the question on the existence of Poisson connections with
Hamiltonian curvature was discussed in \cite{BrF-08}. Our purpose is to study
the set $\operatorname{Conn}_{H}(M,\mathcal{F},P)$ (also the subset of admissible connections) under the symmetry
hypothesis that there exists an action $\Phi$ on $M$ of a compact and
connected Lie group $G$ which preserves the foliation $\mathcal{F}$. In this
situation, one can average the connections on the foliated manifold and the
natural question is to characterize the $G$-actions for which the averaging
procedure preserves the set $\operatorname{Conn}_{H}(M,\mathcal{F},P)$ and the subset of admissible Poisson connections.
Such result for canonical actions with momentum map on general Poisson fiber
bundles was originally stated in \cite{MaMoRa-90} and then extended to locally
Hamiltonian actions on Poisson foliations in \cite{VaVo-14}. In the present paper,
we observe that the averaging procedure preserves the Poisson connections with
Hamiltonian curvature and admissible connections for a more wide class of $G$-actions which admit a
 pre-momentum map in sense of \cite{Ginz-1}. Our main application is that,
starting with an admissible connection $\gamma\in\operatorname{Conn}_{H}(M,\mathcal{F}%
,P)$, we show how to construct a family of $G$-invariant Dirac structures on
$(M,\mathcal{F},P)$ parametrized by 2-cocycles of the Casimir-de Rham complex
associated to the Poisson foliation \cite{AEYu-17}. 

The paper is organized as follows. In Section 2, we recall some basic facts about Poisson connections on regular Poisson foliated manifolds. In Section 3, we  describe an averaging procedure for vector valued forms and connections relative to a foliation preserving action of a compact and connected Lie group G. The basic fact here is that the averaging operator preserve the set of Poisson connections.The notion of a generalized Hannay-Berry connection is introduced in Section 4. We show that such class of Poisson connections naturally appears on a Poisson foliation equipped with a G-action admitting a pre-momentum map preserving each leaf of the foliation, we prove that the difference of a Poisson connection and its averaging takes values in the Hamiltonian vector field (Theorem \ref{THM1}). This fact is a generalization of the results on Hannay-Berry connections on fibrations mentioned before. The main consequence of this result is presented in Theorem \ref{theoavercon} which states that the set of Poisson connections with a Hamiltonian form of curvature $\operatorname{Conn}_{H}(M,\mathcal{F},P)$ and the subset of admissible connections are preserved under the averaging procedure with respect to the canonical action with pre-momentum map. Finally, in Section 5, we apply Theorem \ref{theoavercon} to the construction of families of G-invariant Dirac structures. Here the main results are presented in Theorem \ref{teodiracinv}.

\section{Poisson Connections on Foliated Manifolds}

Let $(M,\mathcal{F})$ be regular foliated manifold and $\mathbb{V}=T$
$\mathcal{F}\subset TM$ the tangent bundle called the\textit{ vertical
distribution}. A vector valued 1-form $\gamma\in\Omega^{1}(M;\mathbb{V})$ is
said to be a \textit{connection} on $(M,\mathcal{F})$ if the vector bundle
morphism $\gamma:TM\rightarrow\mathbb{V}$ satisfies the conditions:%
\begin{equation}
\gamma\circ\gamma=\gamma\text{ and }\operatorname{Im}\gamma=\mathbb{V}.
\label{DC1}%
\end{equation}
In fact, these conditions are equivalent to the following
\[
Y\in\Gamma(\mathbb{V})\Longrightarrow\gamma(Y)=Y.
\]
Then, $\mathbb{H}=\mathbb{H}^{\gamma}:=\ker\gamma$ is \ a normal bundle of
$\mathcal{F}$, called the \textit{horizontal subbundle} (with respect to the
leaf space $M\diagup\mathcal{F}$). It is clear that $\operatorname{id}%
_{TM}-\gamma$ is just the projection to $\mathbb{H}$ along $\mathbb{V}$.

On the contrary, given a normal bundle $\mathbb{H}$ of $\mathcal{F}$, one can
define the associated connection as the projection $\gamma=\gamma^{\mathbb{H}%
}:=\operatorname*{pr}_{2}$ to $\mathbb{V}$ according to the decomposition
\begin{equation}
TM=\mathbb{H}\oplus\mathbb{V}. \label{SP1}%
\end{equation}
Then, the cotangent bundle splits as follows
\begin{equation}
T^{\ast}M=\mathbb{V}^{0}\oplus\mathbb{H}^{0}, \label{SP2}%
\end{equation}
where $\mathbb{V}^{0}$ and $\mathbb{H}^{0}$ are the annihilators of
$\mathbb{V}$ and $\mathbb{H}$, respectively. These decompositions give rise to
$\gamma$-dependent bigrading of differential forms and tensor fields on $M$.
In particular, for any $X\in\mathfrak{X}(M)$ and $\alpha\in\Omega^{1}(M)$, we
have
\[
X=X_{1,0}+X_{0,1}\text{ and }\alpha=\alpha_{1,0}+\alpha_{0,1},
\]
where
\[
X_{1,0}=(\operatorname{id}_{TM}-\gamma)(X)\in\Gamma(\mathbb{H})
\]
and
\[
\alpha_{1,0}=(\operatorname{id}_{T^{\ast}M}-\gamma^{\ast})(\alpha)\in
\Gamma(\mathbb{V}^{0})
\]
are horizontal components and $X_{0,1}$ and $\alpha_{0,1}$ are vertical . Here
$\gamma^{\ast}:T^{\ast}M\rightarrow T^{\ast}M$ is the adjoint of $\gamma$.

Moreover, the exterior differential of forms on $M$ has the following
bigraded decompostion $d=d^\gamma_{1,0}+d^\gamma_{2,-1}+d^\gamma_{0,1}$ associated with decomposition (\ref{SP1}) 
(see, \cite{VaisLecGeoPoisMfds,VaisCoupPoisJAcStr}). The operator $d^\gamma_{1,0}$ is called the \textit{covariant exterior derivative} and it is defined by
\begin{equation}
d^\gamma_{1,0} \alpha (X_0, X_1,\ldots,X_k):= d\alpha ((\operatorname{id}-\gamma)X_0,\ldots,(\operatorname{id}-\gamma) X_k ),\label{covarder}
\end{equation}
for all $\alpha\in \Omega^k(M)$ and all $X_0,X_1,\ldots, X_k \in \Gamma(TM)$. In general, the covariant exterior derivative is not a coboundary operator.

The \textit{ curvature} of a connection
$\gamma$ on $(M,\mathcal{F})$ is a vector valued 2-form $\operatorname{Curv}^{\gamma}\in\Omega
^{2}(M;\mathbb{V})$ on $M$ given by
\begin{equation}
\operatorname{Curv}^{\gamma}:=\frac{1}{2}[\gamma,\gamma]_{\operatorname*{FN}},
\label{FN1}%
\end{equation}
here $[\cdot,\cdot]_{\operatorname*{FN}}$ denotes the Fr\"{o}licher-Nijenhuis bracket
\cite{Kolar}. \ Denote the space of all \textit{projectable vector fields} on
$(M,\mathcal{F})$ by
\[
\mathfrak{X}_{\operatorname*{pr}}(M,\mathcal{F})=\{Z\in\mathfrak{X}%
(M)\mid\lbrack Z,\Gamma(\mathbb{V})]\subset\Gamma(\mathbb{V})\}.
\]
The space of all (local) projectable vector fields which are tangent to the
horizontal subbundle of a connection $\gamma$ will be denoted by
$\Gamma_{\operatorname*{pr}}(\mathbb{H}^{\gamma})$. It follows from 
(\ref{FN1})
that
\begin{equation}
\operatorname{Curv}^{\gamma}(Z_{1},Z_{2})=\gamma([Z_{1},Z_{2}])\label{curvconn}
\end{equation}
for any $Z_{1},Z_{2}\in\Gamma_{\operatorname*{pr}}(\mathbb{H}^{\gamma})$.

It is well-known that the set of all connections is an affine space. Indeed,
fixing a connection $\gamma$ on $(M,\mathcal{F})$, it is easy to see that any
other connection $\tilde{\gamma}$ is of the form
\begin{equation*}
\tilde{\gamma}=\gamma-\Xi,
\end{equation*}
where the vector bundle morphism $\Xi:TM\rightarrow TM$ is called the
\textit{connection difference form} and satisfies the conditions%
\[
\operatorname{im}\Xi\subseteq\mathbb{V\subseteq}\ker\Xi\text{.}%
\]
The horizontal subbundle associated to $\tilde{\gamma}$ is given by
\[
\mathbb{H}^{\tilde{\gamma}}=\left(  \operatorname{Id}+\Xi\right)
(\mathbb{H}^{\gamma}).
\]
and hence
\[
\Gamma_{\operatorname*{pr}}(\mathbb{H}^{\tilde{\gamma}})=\{\tilde{Z}%
=Z+\Xi(Z)\mid Z\in\Gamma_{\operatorname*{pr}}(\mathbb{H}^{\gamma})\}.
\]
It follows from here and (\ref{FN1}) that we have the following transition rule for
the curvature:%
\begin{align}
\operatorname{Curv}^{\tilde{\gamma}}(Z_{1},Z_{2}) =  & \operatorname{Curv}%
^{\gamma}(Z_{1},Z_{2})+[\Xi(Z_{1}),\Xi(Z_{2})]+\label{CV1}\\
&  \lbrack\Xi(Z_{1}),Z_{2}]-[\Xi(Z_{2}),Z_{1}]-\Xi([Z_{1},Z_{2}])\nonumber
\end{align}
for $Z_{1},Z_{2}\in\Gamma_{\operatorname*{pr}}(\mathbb{H}^{\gamma})$.

A \textbf{\textit{Poisson foliation}} is a triple $(M,\mathcal{F},P)$ consisting of a regular foliated manifold $(M,\mathcal{F})$ equipped with a vertical Poisson bivector
field $P\in\Gamma(\wedge^{2}\mathbb{V}),$ $[P,P]_{\operatorname*{SCH}}=0$. Thus, the Poisson structure $P$ is characterized by the property: every
symplectic leaf of $P$ belongs to the leaf of $\mathcal{F}$.

A \textit{connection} $\gamma$ is said to be \textit{Poisson} on
$(M,\mathcal{F},P)$ if every (local) $\mathbb{H}^{\gamma}$-tangent projectable
vector field $Z\in\Gamma_{\operatorname*{pr}}(\mathbb{H}^{\gamma})$ is Poisson
on $(M,P)$, that is, $L_{Z}P=0$. In this case, for every $\operatorname{Curv}%
^{\gamma}(Z_{1},Z_{2})$ is a vertical Poisson vector field, for every
$Z_{1},Z_{2}\in\Gamma_{\operatorname*{pr}}(\mathbb{H}^{\gamma})$.



\section{The Averaging Procedure}

First, we recall the averaging procedure for connections on a regular foliated
manifold $(M,\mathcal{F})$ .

Let $G$ be a \textit{compact and connected} Lie group and $\mathfrak{g}$ its Lie algebra. Suppose that we are
given an action $\Phi:G\times M\rightarrow M$ of $G$ which \textit{preserves the foliation} $d_{m}\Phi_{g}\mathbb{V}_{m}=\mathbb{V}_{\Phi_{g}(m)}$,
$\forall g\in G$. Equivalently
\begin{equation}
Y\in\Gamma(\mathbb{V})\Longrightarrow\Phi_{g}^{\ast}Y\in\Gamma(\mathbb{V}%
).\label{FP}%
\end{equation}
For every $\xi\in\mathfrak{g}$, the corresponding infinitesimal generator of the
$G$-action is denoted by $\xi_{M}$,
\[
\xi_{M}(p):=\left.\frac{d}{dt}\right|_{t=0}\Phi_{\exp(t\xi)}(p),\text{ \ }p\in M.
\]
Condition (\ref{FP}) implies that each infinitesimal generator is a projectable
vector field,%
\[
\xi_{M}\in\mathfrak{X}_{\operatorname*{pr}}(M,\mathcal{F})\text{ \ }\forall
a\in\mathfrak{g}.
\]
As a consequence, the $G$-action preserves the space of all projectable vector
fields%
\[
\Phi_{g}^{\ast}(\mathfrak{X}_{\operatorname*{pr}}(M,\mathcal{F}))=\mathfrak{X}%
_{\operatorname*{pr}}(M,\mathcal{F}).
\]
Let $\Omega^{k}(M;TM)$ the space of vector valued $k$-form. For any $K\in\Omega^{k}(M;TM)$, the  $G$-\textit{average} of $K$ is the vector valued for
$\langle K\rangle ^{G}\in\Omega^{k}(M;TM)$ defined by the standard formula:
\[
\bar{K}=\langle K\rangle ^{G}:=\int_{G}\Phi_{g}^{\ast}Kdg
\]
Here, the pull-back $\Phi_{g}^{\ast}K$ of $K$ is given by
\[
(\Phi_{g}^{\ast}K)(Y_{1},...,Y_{k})=\Phi_{g}^{\ast}(K(\Phi_{g}{}_{\ast}%
Y_{1},...,\Phi_{g}{}_{\ast}Y_{k}))
\]
for $Y_{1},...,Y_{k}\in\mathfrak{X}(M)$ and the integral is taken with respect to the normalized Haar measure $dg$ is on $G$, $\int_{G}dg=1$. 

Recall that a vector valued $k$-form $K$ is said to be $G$-\textit{invariant} if
$\Phi_{g}^{\ast}K=K$ $\forall g\in G$. Since the group $G$ is connected, this invariance condition can be represented in the infinitesimal terms: $L_{\xi_{M}%
}K=0$  $\forall a\in\mathfrak{g}$. It is clear that the $G$-average $\langle K\rangle ^{G}$ is $G$-invariant for any $K$. 
We have the following invariance criterion: $K$ is $G$-invariant if and only if $\langle K\rangle ^{G}=K$.

Property (\ref{FP}) implies that the averaging operator preserves the set of all connections. In other words, for any connection $\gamma$ on
$(M,\mathcal{F})$, its $G$-average $\bar{\gamma}=\langle\gamma\rangle ^{G}$is a
$G$-invariant vector valued 1-form which again satisfies the conditions in (\ref{DC1}). From the property that the Fr\"{o}licher-Nijenhuis bracket is a natural
operation with respect to the pull-back, it follows that the curvature form of
$\bar{\gamma}$ is also $G$-invariant,
\begin{equation}
\langle\operatorname{Curv}^{\bar{\gamma}}\rangle ^{G}=\operatorname{Curv}^{\bar{\gamma}}.
\label{IK}%
\end{equation}
Indeed,
\[
\Phi_{g}^{\ast}\operatorname{Curv}^{\bar{\gamma}}=\frac{1}{2}\Phi_{g}^{\ast
}[\gamma,\gamma]_{\operatorname*{FN}}=\frac{1}{2}[\Phi_{g}^{\ast}\bar{\gamma
},\Phi_{g}^{\ast}\bar{\gamma}]_{\operatorname*{FN}}=\operatorname{Curv}%
^{\bar{\gamma}}%
\]

Now, consider the connection difference form%
\begin{equation}
\Xi^{G}:=\gamma-\langle\gamma\rangle ^{G}\in\Omega^{1}(M;\mathbb{V}). \label{CD1}%
\end{equation}

\begin{lemma}\label{lemmrepre}
We have the following representation%
\begin{equation}
\Xi^{G}=\int_{G}\int_{0}^{1}\Phi_{\exp(ta)}^{\ast}[\gamma,\xi_{M}
]_{\operatorname*{FN}}\ dtdg \label{CD2}%
\end{equation}
\end{lemma}
\begin{proof}
By the fundamental theorem of calculus, we obtain
\begin{equation}\label{equationXi1}
\Phi^*_{\exp(a)}\gamma-\gamma=
\int_0^1 \Phi^*_{\exp(ta)}
L_{\xi_M}\gamma\ d t.
\end{equation}
Integrating the equality (\ref{equationXi1}) with respect
to the Haar measure, we get
$$\langle\gamma \rangle _G-\gamma=\int_G\int_0^1
\Phi^*_{\exp(ta)}L_{\xi_M}\gamma\ 
d t\ d g,$$
From (\ref{CD1}) and the identity $L_{\xi_M}\gamma=-[\gamma,\xi_M]_{FN}$, it follows that
\begin{equation}
\Xi^G=\int_G\int_0^1\Phi^*_{\exp(ta)}[\gamma,\xi_M]_{FN} d t d g.\label{diffaver}
\end{equation}
\end{proof}

For connections on foliated manifolds, we also have the following invariance criteria.

\begin{proposition}\label{propinvconn}
For a given connection $\gamma$ on $(M,\mathcal{F})$ and a foliation
preserving action $\Phi:G\times M\rightarrow M$ of a compact connected Lie
group $G$ , the following conditions are equivalent:

\begin{itemize}

\item[(i)] $\gamma$ is $G$-invariant;

\item[(ii)] $\langle\gamma\rangle ^{G}=\gamma$;

\item[(iii)] for every $a\in\mathfrak{g}$
\[
[\gamma,\xi_{M}]_{\operatorname*{FN}}=0;
\]

\item[(iv)] the horizontal distribution $\mathbb{H}^{\gamma}$ is $G$-invariant,%
\[
d_{m}\Phi_{g}(\mathbb{H}_{m}^{\gamma})=\mathbb{H}_{\Phi_{g}(m)}^{\gamma}\forall g\in G;
\]

\item[(v)] the connection difference form $\Xi^{G}$ is zero.

\end{itemize}
\end{proposition}

\begin{proof}
The equivalence between $(i)$ and $(ii)$ follows by straight forward computations. The implication (i) $\Rightarrow$(iii) follows from the relations:
\begin{equation}
\Phi^*_{\exp(t\xi)}[\gamma,\xi_M]_{FN}=
-\Phi^*_{\exp(t\xi)}L_{\xi_M}\gamma=-\left.\frac{d}{d t}
\right.\Phi^*_{\exp(t\xi)}\gamma .\label{frolbra}
\end{equation}
Conversely, condition (iii) together with \eqref{frolbra} and the connectness of G imply the invariance condition (i). Here we use the fact \cite{WusLieGpSqExp}: every element of a connected Lie group is the product of $\exp(\xi_1)$ and $\exp(\xi_2)$ for some $\xi_1,\xi_2\in\mathfrak{g}$. The equivalence between $(i)$ and $(iv)$ follows from the fact that, the $G$-invariance of $\gamma$ is equivalent to the equation
$$(\gamma)_{\Phi_g(m)}
\circ T_m\Phi_g=T_m\Phi_g\circ (\gamma)_{m}.$$
Finally, the equivalence between $(ii)$ and $(v)$ follows directly from (\ref{CD1}).
\end{proof}

Next, we formulate some key properties of the averaged connection in the case of the leaf tangent $G$-action.

\begin{lemma}\label{lemleaftang}
Assume that the $G$-action on $(M,\mathcal{F})$ is leaf tangent,%
\begin{equation}
\xi_{M}\in\Gamma(\mathbb{V})\text{ }\forall \xi\in\mathfrak{g}.\label{LT}%
\end{equation}
Then, for every connection $\gamma$ on $(M,\mathcal{F})$ the following
assertions hold:
\begin{itemize}

\item[(a)] $\gamma$ is $G$-invariant if and only if
\begin{equation}
[ Z,\xi_{M}]=0 \ \ \forall Z\in\Gamma_{\operatorname*{pr}}(\mathbb{H}^{\gamma}),\ \xi\in\mathfrak{g}.\label{condinfgen}
\end{equation}

\item[(b)] The space of horizontal projectable vector fields associated with the
averaged connection $\bar{\gamma}=$ $\langle\gamma\rangle ^{G}$ is described as
\begin{equation}\label{HG1}
\Gamma_{\operatorname*{pr}}(\mathbb{H}^{\bar{\gamma}})=\{\langle Z\rangle ^{G}\mid
Z\in\Gamma_{\operatorname*{pr}}(\mathbb{H}^{\gamma})\}.
\end{equation}

\item[(c)] For all $Z\in\Gamma_{\mathrm{pr}}(\mathbb{H}^{\gamma})$,
\begin{equation}\label{EX1}
\displaystyle\Xi^G(Z)=-\int_G \int_0^1\Phi^*_{\exp(t\xi)}[Z,\xi_M]dt dg.
\end{equation}
\end{itemize}
\end{lemma}
\begin{proof}
\begin{itemize}
\item[(a)] 
For each $Z\in\Gamma_{\mathrm{pr}}(\mathbb{H}^{\gamma})$, we have
$$[\gamma,\xi_M]_{\mathrm{FN}}(Z)=[\gamma(Z),\xi_M]
-\gamma([Z,\xi_M])=\gamma([Z,\xi_M]),$$
Then, from here and  Proposition \ref{propinvconn}, it follows that the G-invariance of $\gamma$ is equivalent to the condition that $[Z,\xi_M]$ is a horizontal vector field. If the action is leaf tangent, then the vector field  $[Z,\xi_M]$ is always vertical and hence equals zero. Conversely, if condition \eqref{condinfgen} holds, then $[\gamma,\xi_M]_{\mathrm{FN}}(Z)=0$ $\forall Z\in\Gamma_{\operatorname*{pr}}(\mathbb{H}^{\gamma})$ and $[\gamma,\xi_M]_{\mathrm{FN}}(V)=0$ for each $V\in\Gamma(\mathbb{V})$.

\item[(b)] Each vector field $\widetilde{Z}\in\Gamma_{\mathrm{pr}}
(\mathbb{H}^{\bar{\gamma}})$, is of the form 
$\widetilde{Z}=Z+\Xi^G(Z)$
with $Z\in\Gamma_{\mathrm{pr}}(\mathbb{H}^{\gamma})$.
Moreover $\widetilde{Z}$ is a $G$-invariant vector field 
by the item (a). From here and the fact that the average of $\Xi^G$ is zero average, we get that
$$\widetilde{Z}=\langle \widetilde{Z}\rangle ^G=\langle Z\rangle ^G.$$
This proves (\ref{HG1}).

\item[(c)] For every $Z\in\Gamma_{\mathrm{pr}}(\mathbb{H}^{\gamma})$,
and the $G$-invariant vector field
$\widetilde{Z}:=Z+\Xi^G(Z)\in\Gamma_{\mathrm{pr}}
(\mathbb{H}^{\bar{\gamma}})$ it follows from formula \eqref{diffaver} that
\begin{equation}\label{equach4auxXi1}
\begin{array}{rcl}
\Xi^G(Z)=\Xi^G(\widetilde{Z})&=&\int_G\int_0^1
\left(\Phi^*_{\exp(t\xi )}[\gamma,\xi_M]_{FN}
\right)\left(\widetilde{Z}\right)  dt dg,
\vspace{2mm}\\
&=&\displaystyle\int_G\int_0^1
\Phi^*_{\exp(t\xi )}\left([\gamma,\xi_M]_{FN}
\left(\widetilde{Z}\right)\right) d t d g.
\end{array}
\end{equation}
Then, formula \eqref{EX1} follows from \eqref{equach4auxXi1} and the equality:
\begin{equation}\label{equach4auxXi2}
[\gamma,\xi_M]_{FN}(\widetilde{Z})
=[\gamma,\xi_M]_{FN}(Z).
\end{equation}
\end{itemize}
\end{proof}

Now, let us turn to the Poisson case. The following result states conditions under which the averaging of a Poisson connection inherits the property of being Poisson. 

\begin{lemma}
Let $(M,\mathcal{F},P)$ be a Poisson foliation. Suppose that the $G$-action is leaf tangent (condition (\ref{LT})) and canonical relative to $P$,%
\[
L_{\xi_{M}}P=0 \ \forall a\in\mathfrak{g}.%
\]
Then, the $G$-average $\bar{\gamma}$ of every Poisson connection $\gamma$ on
$(M,\mathcal{F},P)$ is again Poisson. Moreover, the curvature of
$\bar{\gamma}$ has the following property: if $Z_{1},Z_{2}%
\in\Gamma_{\operatorname*{pr}}(\mathbb{H}^{\gamma})$ then 
$\operatorname{Curv}^{\bar{\gamma}}(Z_{1},Z_{2})$ is a 
Poisson vertical G-invariant vector field.
\end{lemma}
\begin{proof}
Taking into account that the action is canonical and $\gamma$ is a Poisson connection, by standard properties of the averaging operator we obtain
 $$0=\left\langle L_Z P\right\rangle _G=L_{\left\langle Z\right\rangle _G}P,$$
for all $Z\in\Gamma(\mathbb{H}^{\gamma})$, that is, the average of a $\gamma$-horizontal projectable vector field is Poisson. Under the assumption that the action is leaf tangent, point (b) of Lemma \ref{lemleaftang} implies that the $\bar{\gamma}$-horizontal projectable vector fields are Poisson and hence, $\bar{\gamma}$ is a Poisson connection. The last assertion of the lemma follows directly from \eqref{curvconn}.

\end{proof}

\section{Generalized Hannay-Berry Connections}

Let $(M,\mathcal{F},P)$ be a Poisson foliation. Suppose we are given an action
$\Phi:G\times M\rightarrow M$ of a connected, compact Lie group $G$ which
admits a \textit{pre-momentum map} in the sense of 
\cite{Ginz-1}, that is, there
exists a linear map $\mu:\mathfrak{g}\rightarrow\Omega^{1}(M)$ such that
\begin{equation}
\xi_{M}=P^{\sharp}\mu_{\xi},\label{C1}%
\end{equation}
where
\begin{equation}
\mathbf{i}_{P^{\sharp}\alpha}d\mu_{\xi}=0\label{C2}%
\end{equation}
for all $\xi\in\mathfrak{g}$. This condition means that the pull-back of
the 1-form $\mu_{\xi}$ to each symplectic leaf of $P$ is closed.

\begin{remark}
The notion of a pre-momentum map, introduced by V. Ginzburg in \cite{Ginz-1}, in general, involves only condition (4.1) which says that the infinitesimal generators are tangent to the symplectic foliation of $P$. Property \eqref{C2} appears in \cite{Ginz-1} as an extra condition under the study of some problems related to equivariant Poisson cohomology.

\end{remark}
Note also that conditions (\ref{C1}), (\ref{C2}) imply that the $G$-action is canonical on  $(M,P)$. Indeed, for every 
$\alpha,\beta\in\Omega^1(M)$
\[
(L_{\xi_{M}}P)(\alpha,\beta)=(L_{P^{\sharp}\mu_{\xi}}P)(\alpha,\beta)=d\mu
_{\xi}(P^{\sharp}\alpha,P^{\sharp}\beta)=0.
\]

For every $\beta\in\Gamma(\bigwedge^q\mathbb{V}^0)$ and $Q\in\Gamma(\bigwedge^1\mathbb{V}^0)$,
denote by $\{Q\wedge\beta\}_P$ the element of $\Gamma(\bigwedge^{q+1}
\mathbb{V}^0)$ given by
$$\{Q\wedge\beta\}_P(Z_0,Z_1,...,Z_q):=\sum_{i=0}^{q}(-1)^i
\{Q(Z_i),\beta(Z_0,Z_1,...,\hat{Z_i},...,Z_q)\}_P$$
where $\{\cdot,\cdot\}$ is the Poisson bracket associated to $P$.

\begin{theorem}\label{THM1}
Under the assumptions (\ref{C1}), (\ref{C2}), for any Poisson connection $\gamma$ on
$(M,\mathcal{F},P)$, the connection difference form $\Xi^{G}$ $=\gamma
-\langle\gamma\rangle^{G}$ takes values in Hamiltonian vector fields of the vertical
Poisson structure $P$,
\begin{equation}
\Xi^{G}(Z)=P^{\sharp}dQ(Z) \  \forall Z\in\Gamma_{\operatorname*{pr}%
}(\mathbb{H}^{\gamma}),\label{MR1}%
\end{equation}
where $Q\in\Gamma(\mathbb{V}^{0})$ is a horizontal 1-form defined by
\begin{equation}
Q(Z):=-\int_{G}\int_{0}^{1}\Phi_{\exp(ta)}^{\ast}\mathbf{i}_{Z}(\mu_{a})_{1,0}dtdg.\label{MR2}%
\end{equation}
The curvature of the averaged connection $\bar{\gamma}=\langle \gamma\rangle ^{G}$ is given by%
\begin{equation}\label{CV2}
\begin{array}{rcl}
\displaystyle
\operatorname{Curv}^{\langle \gamma\rangle ^{G}}(Z_{1},Z_{2})&=&\operatorname{Curv}^{\gamma
}(Z_{1},Z_{2})+\vspace{2mm}\\
&&\displaystyle P^{\sharp}d\left(  d^\gamma_{1,0}Q(Z_{1},Z_{2})
+\frac{1}{2}\{Q\wedge Q\}_{P}(Z_1,Z_2)\right),
\end{array}
\end{equation}
for all $Z_{1},Z_{2}\in\Gamma_{\operatorname*{pr}}(\mathbb{H}^{\gamma})$
\end{theorem}
\begin{proof}
Combining item (c) of Lemma \ref{lemleaftang} and condition (\ref{C1}), we obtain
\begin{equation*}
\Xi^G(Z)=
-\int_G\int_0^1 \Phi^*_{\exp(t\xi )}
L_ZP^{\sharp}(\mu_{a}) d t d g,
\end{equation*}
where $Z\in\Gamma_{\mathrm{pr}}(\mathbb{H}^{\gamma})$. By using relations \eqref{C1} and \eqref{C2}, we get that
$$\Xi^G(Z)=-P^{\sharp}d\left(\int_G\int_0^1 
\Phi^*_{\exp(t\xi )} \mathbf{i}_Z\mu_{a}\ 
d t\ d g\right).$$
Taking into account that $\mathbf{i}_Z(\mu_{a})=\mathbf{i}_Z(\mu_{a})_{1,0} $ for $Z\in\Gamma_{\mathrm{pr}}(\mathbb{H}^{\gamma})$, we verify \eqref{MR1}.
Now, from (\ref{MR1}), we obtain the following identities
\begin{equation}\label{equch4auxRcon1} 
    \begin{split} 
    \Xi^G([Z_1,Z_2])=P^{\sharp}d (Q([Z_1,Z_2])), \ \ \ [\Xi^G(Z_1),Z_2]&=-P^{\sharp}d(L_{Z_2}Q(Z_1)), \\
\  \text{and}  \ \ \ \ \ [\Xi^G(Z_2),Z_1]=-P^{\sharp}d(L_{Z_1}Q(Z_2)),
  \end{split} 
\end{equation} 
for all $Z_1,Z_2\in\Gamma_{\mathrm{pr}}(\mathbb{H}^{\gamma})$. Moreover,
\begin{equation}\label{equch4auxRcon4}
[\Xi^G(Z_1),\Xi^G(Z_2)]=P^{\sharp}d \left(\{Q(Z_1),Q(Z_2)\}_P\right).
\end{equation}
These relations together with \eqref{CV1} imply \eqref{CV2}.

\end{proof}

\begin{corollary}\label{genhordist}
\bigskip The horizontal distribution of the averaged connection $\bar{\gamma}$
is generated by the $G$-invariant Poisson vector fields of the form
\begin{equation}
\langle Z\rangle ^{G}=Z+P^{\sharp}dQ(Z),\label{dishorgamma}
\end{equation}
where $Z$ runs over $\Gamma_{\operatorname*{pr}}(\mathbb{H}^{\gamma})$.
\end{corollary}

\begin{remark}
In the context of the Poisson cohomology of $(M,P)$, one can derive from Corollary \ref{genhordist} the following fact \cite{MYu-17}: for every $\gamma
$-horizontal $k$-cocycle $A\in\Gamma(\wedge^{k}\mathbb{H}^{\gamma}),$
$[P,A]_{\operatorname*{SCH}}=0$ its Poisson cohomology class is represented by
a $G$-invariant $k$-tensor. This partially recovers the results on the
equivariant Poisson cohomology due to \cite{Ginz-1}. 

\end{remark}

Now, let us consider some special cases. It is clear that conditions (\ref{C1}),(\ref{C2}) hold in the case when the $G$-action is \textit{locally Hamiltonian} on $(M,P)$, that is,
\[
d\mu_{a}=0\ \forall a\in\mathfrak{g.}%
\]
In particular, in the standard case \cite{MaMoRa-90} of a \textit{Hamiltonian}  $G$-\textit{action with momentum map }
$\mathbb{J}:M\rightarrow\mathfrak{g}^{\ast}$,%
\[
\xi_{M}=P^{\sharp}d\mathbb{J}_{\xi},
\]
formula (\ref{MR2}) for the horizontal 1-form $Q$ reads
\[
Q=-\int_{G}\int_{0}^{1}\Phi_{\exp(t\xi )}^{\ast}(d_{1,0}^{\gamma}\mathbb{J}%
_{\xi})dtdg.
\]




Theorem \ref{THM1} presents a generalized version of the results on Hannay-Berry connections obtained in \cite{MaMoRa-90} in the case of a Poisson fiber bundle equipped with Hamiltonian $G$-action with momentum map. Thus, in the case of a $G$-action with pre-momentum map $\mu$ on a Poisson foliation $(M,\mathcal{F},P)$, the averaged Poisson connection  
$\bar{\gamma}=\langle\gamma\rangle^{G}$ can be called a \textit{generalized Hannay-Berry connection}.

\section{Poisson connections with Hamiltonian  Curvature}
 Starting with a Poisson foliation $(M,\mathcal{F},P)$, denote by $\operatorname{Con}_{\operatorname{H}}(M,\mathcal{F},P)$ the set of all Poisson connections
$\gamma$ on the Poisson foliation whose curvature form takes values in the
space of Hamiltonian vector fields of the vertical Poisson structure $P$,%
\begin{equation}
\operatorname{Curv}^{\gamma}(Z_{1},Z_{2})=-P^{\sharp}d\sigma^{\gamma}%
(Z_{1},Z_{2})\text{ \ }\forall Z_{1},Z_{2}\in\Gamma_{\operatorname*{pr}%
}(\mathbb{H}^{\gamma}),\label{HAM1}%
\end{equation}
for a certain horizontal 2-from $\sigma^{\gamma}\in\Gamma(\wedge^{2}%
\mathbb{V}^{0})$ which is called a \textit{Hamiltonian form of the curvature}.

Denote by $\mathcal{C}^{k}=\mathcal{C}^{k}(M,\mathcal{F},P)$ the space of all
horizontal $k$-forms $\beta\in\Gamma(\wedge^{k}\mathbb{V}^{0})$ which take
values in the space $\operatorname{Casim}(M,P)$ of Casimir functions of $P$,%
\[
\beta(X_{1},...,X_{k})\in\operatorname{Casim}(M,P)\text{ \ \ }\forall X_{i}%
\in\mathfrak{X}_{\operatorname*{pr}}(M,\mathcal{F}).
\]
Then, it is clear that a Hamiltonian form $\sigma^{\gamma}$ of the curvature
in (\ref{HAM1}) is defined up to the transformations
\begin{equation}
\sigma^{\gamma}\mapsto\sigma^{\gamma}+C \ \ \forall\  C\in\mathcal{C}%
^{2}. \label{FR}%
\end{equation}

In particular, if $\sigma^\gamma \in \mathcal{C}^2$ then the connection is flat and the covariant exterior derivative $d^\gamma_{1,0}$ is a coboundary operator.
\begin{definition}
A Poisson connection $\gamma\in\operatorname{Con}_{\operatorname{H}%
}(M,\mathcal{F},P)$ is said to be admissible if there exists a Hamiltonian form
$\sigma=\sigma^{\gamma}\in\Gamma(\wedge^{2}\mathbb{V}^{0})$ of the
curvature in \eqref{MF1} which satisfies the the $\gamma$-covariant constancy condition condition
\[
d_{1,0}^{\gamma}\sigma=0.
\]
\end{definition}

Notice that, in general, for a given $\gamma\in\operatorname{Con}_{\operatorname{H}}(M,\mathcal{F},P)$, by the \textit{Bianchi identity}, we have $d_{1,0}^{\gamma}\sigma\in\mathcal{C}^{3}$. Moreover, $d^{\gamma}_{1,0} (\mathcal{C}^k)\subset \mathcal{C}^{k+1}.$ Hence, one can define the operator $\bar
{d}^{\gamma}:\mathcal{C}^{k}\rightarrow\mathcal{C}^{k+1}$  just by $\bar
{d}^{\gamma} = \left.d_{1,0}^{\gamma}\right|_{\mathcal{C}^{k}}$ which results to be a coboundary operator. Thus, one can associate to the setup $(M,\mathcal{F},P,\gamma)$ the cochain complex $(\oplus_{k=0}^{\infty
}\mathcal{C}^{k},\bar{d}^{\gamma})$ called the foliated de Rham-Casimir complex, \cite{Vo-01,Vo-05,AEYu-17}. Taking into
account that the freedom in the choice of $\sigma^{\gamma}$ is given by the
transformation (\ref{FR}), we derive the following 
criterion for $\gamma$ to be
admissible: $d_{1,0}^{\gamma}\sigma$ is a 3-cocycle relative to $\bar
{d}^{\gamma}$ and its cohomology class is trivial.

Now, suppose that we are given an action on $M$ of a connected, compact Lie group $G$ with a pre-momentum map $\mu$.  Since all infinitesimal generators $\xi_{M}$ of the $G$-action are tangent to
the symplectic foliation of $P$, we have
\[
k\in\operatorname{Casim}(M,P)\Longrightarrow L_{\xi_{M}}k=0 \ \forall
a\in\mathfrak{g}.%
\]
and hence any horizontal 2-form $C\in\mathcal{C}^{2}$ is $G$-invariant,
$L_{\xi_{M}}C=0$. It follows that the $G$-invariance of a Hamiltonian form
$\sigma^{\gamma}$ is preserved under transformation (\ref{FR}).

Since the $G$-action is canonical relative to $P$ and preserves the vertical
distribution $\mathbb{V}$, it is easy to see that the group $G$ naturally acts on the set of Poisson
connections $\operatorname{Con}_{\operatorname{H}}(M,\mathcal{F},P)$,
$\gamma\mapsto\Phi_{g}^{\ast}\gamma$. Moreover, as a consequence of 
Theorem \ref{THM1}, we get the following fact.

\begin{theorem}\label{theoavercon}
The averaging procedure with respect to the canonical action $\Phi:G\times
M\rightarrow M$ \ with a pre-momentum map $\mu$ preserves the set
$\operatorname{Con}_{\operatorname{H}}(M,\mathcal{F},P)$, that is,%
\[
\gamma\in\operatorname{Con}_{\operatorname{H}}(M,\mathcal{F},P)\Longrightarrow
\ \langle \gamma\rangle \ \in\operatorname{Con}_{\operatorname{H}}(M,\mathcal{F},P).
\]
where the Hamiltonian form of the curvature of $\langle \gamma\rangle $ is given by
\begin{equation}
\bar{\sigma}=\sigma^{\langle \gamma\rangle }:=\sigma^{\gamma}-\left(  d_{1,0}^{\gamma
}Q+\frac{1}{2}\{Q\wedge Q\}_{P}\right)  .\label{MF1}%
\end{equation}
and the horizontal 1-form $Q\in\Gamma(\mathbb{V}^{0})$ is defined in terms of
$\mu$ by formula (\ref{MR2}). Moreover, if $\gamma$ is admissible so also $\langle \gamma\rangle $; that is,
\[
d_{1,0}^{\gamma}\sigma=0 \Longrightarrow
d_{1,0}^{\langle \gamma\rangle }\bar{\sigma}=0.
\]
\end{theorem}
\begin{proof}
The first part of this result is a direct consequence of Theorem \ref{THM1}. In particular, the formula for the Hamiltonian form of $\langle \gamma\rangle $ follows from equation \eqref{CV2}. So, it remains to prove that the averaging procedure preserves the admissibility property. Assume that $d_{1,0}^{\gamma}\sigma=0.$
Since $\sigma\in\Gamma(\bigwedge^2\mathbb{V}^0)$,
the relation (\ref{CD1}) implies the formula 
\begin{equation}\label{EQA1}
d^{\bar{\gamma}}_{1,0}
\sigma=d^{\gamma}_{1,0}\sigma+\{Q\wedge\sigma\}_P.
\end{equation}
Recall that the exterior differential has the following bigraded decomposition $d= d^\gamma_{1,0}+d^\gamma_{0,1}+d^\gamma_{2,-1}$ depending on the connection $\gamma$. Taking account that $d^2=0$, we obtain the following identity 
$(d^{\gamma}_{1,0})^2=-[d_{0,1}^{\gamma},d_{2,-1}^{\gamma}]$.
In particular, for $Q\in\Omega^{1,0}(M)$, 
the equation (\ref{HAM1}) implies that
\begin{equation}\label{EQA2}
(d_{1,0}^{\gamma})^2Q= \{Q\wedge\sigma\}_P.
\end{equation}
On the other hand, since $\gamma$ is a Poisson connection, 
we obtain, by straightforward computation that
\begin{equation}\label{EqA3}
d_{1,0}^{\gamma}\frac{1}{2}\{Q\wedge Q\}_P=-\{Q\wedge d_{1,0}^{\gamma}Q\}_P
\end{equation}
By Theorem \ref{theoavercon} , relations 
(\ref{EQA1}), (\ref{EQA2}) and (\ref{EqA3}), it follows that 
$d_{1,0}^{\bar{\gamma}}\bar{\sigma}=0$.
\end{proof}

\begin{remark}
The $G$-invariance of the curvature $\operatorname{Curv}^{\langle \gamma\rangle }$ implies
only that
\[
L_{\xi_{M}}\bar{\sigma}\in\mathcal{C}^{2} \ \ \forall a\in\mathfrak{g}.
\]
\end{remark}


We end this section by formulating the usefull property of a pre-momentum map. 

\begin{proposition}\label{casimirinhor}
For each $\xi\in \mathfrak{g}$ and $Z\in \Gamma_{\operatorname{pr}}(\mathbb{H}^{\bar{\gamma}})$, $\mathrm{i}_Z\mu^\xi $ is a Casimir function.
\end{proposition}
\begin{proof}
Let $Z\in \Gamma_{\operatorname{pr}}(\mathbb{H}^{\bar{\gamma}})$. Since $Z$ is a $G$- invariant Poisson vector field, it follows that $0=P^\sharp(L_Z \mu^\xi)$, for all $\xi\in \mathfrak{g}$. From this fact and condition \eqref{C2}, we have
\begin{eqnarray*}
0&=& d\mu^\xi(P^\sharp \alpha, Z)=L_{P^\sharp \alpha}(\mu^\xi(Z))-L_Z (\mu^\xi)(P^\sharp \alpha))-\mu^\xi([P^\sharp \alpha,Z]),\\
&=& L_{P^\sharp \alpha}(\mu\xi(Z))-\mu^\xi(P^\sharp( L_Z \alpha))+\mu^\xi([Z.P^\sharp \alpha]),\\
&=& d (\mu^\xi(Z) )(P^\sharp \alpha)= - \alpha (P^\sharp(\mathrm{i}_Z\mu^\xi  )),
\end{eqnarray*}
for every $\alpha \in \Omega^1(M)$. This implies that $\mathrm{i}_Z\mu^\xi $ is a Casimir function.
\end{proof}

\section{Adiabatic condition}

In the previous sections, we dealt with two structures compatibles with a foliated Poisson manifolds $(M, \mathcal{F},P)$:  a Poisson connection $\gamma$ and a $G$- action with pre-mometum map  $\mu :  \mathfrak{g}\rightarrow \Omega^1(M)$. In general, theses two structures are independent. Here.,we will relate these structures by the so-called \textit{adiabatic condition}, \cite{MaMoRa-90}. 

Suppose that a foliated Poisson manifold $(M, \mathcal{F},P)$ equipped with a $G$-action with a pre-momentum map $\mu$ is given. 

\begin{definition}
Given a Poisson conecction $\gamma$ on $(M, \mathcal{F},P)$. We say that  pre-momentum map $\mu$ satisfies the adiabatic condition (relative to $\gamma$) if 
\begin{equation}
\langle(\mathrm{id}_{T^*M}-\gamma^*) \mu^\xi \rangle=0, \ \forall \xi\in \mathfrak{g}. \label{adiacond}
\end{equation}
Here $\gamma^{\ast}:T^{\ast}M\rightarrow T^{\ast}M$ is the dual vector bundle morphism.
\end{definition}
In particular, in the case when the $G$-action on $(M,\mathcal{F},P)$ is
canonical with a momentum map $\mathbb{J}:\mathfrak{g}\rightarrow C^{\infty
}(M)$, we have $\mu^{\xi}=d\mathbb{J}^{\xi}$ and condition \eqref{adiacond} reads%
\[
\langle d_{1,0}^{\gamma}\mathbb{J}^{\xi}\rangle^{G}=0.
\]
This is just the adiabatic condition which was introduced in \cite{MaMoRa-90} for
Hamiltonian actions on Poisson fiber bundles.

The following observation says how to reformulate the adiabatic condition in terms of the averaged connection  $\bar{\gamma}$.

\begin{lemma}\label{adiabaticprop}
Let $\gamma$ be a Poisson connection on $(M,\mathcal{F}, P)$. Then,
\begin{equation}
(\mathrm{id}_{T^*M}-\bar{\gamma}^*) \mu^\xi =\langle(\mathrm{id}_{T^*M}-\gamma^*) \mu^\xi \rangle, \ \forall \xi\in\mathfrak{g}, \label{adiaprop}
\end{equation}
where $\bar\gamma = \langle \gamma \rangle$. Moreover, $\mu$ satisfies the adiabatic condition relative to $\gamma$ if and only if $(\mathrm{id}_{T^*M}-\bar{\gamma}^*) \mu^\xi =0$ for all $\xi\in \mathfrak{g}$.
\end{lemma}
\begin{proof}
Since the 1-forms in both sides of \eqref{adiaprop} vanish on vertical vector fields  as it can be easily prove it, we only need to check the equation \eqref{adiaprop} for horizontal vector fields. Let $Z\in \Gamma_{\operatorname{pr}}(\mathbb{H}^{\bar{\gamma}})$. By Proposition \ref{casimirinhor}, $(\mathrm{id}_{T^*M}-\bar{\gamma}^*) \mu^\xi(Z)=\mu^\xi(Z)$ is a Casimir function and therefore a $G$-invariant function. Thus,
\begin{equation*}
(\mathrm{id}_{T^*M}-\bar{\gamma}^*) \mu^\xi(Z)=\langle(\mathrm{id}_{T^*M}-\bar{\gamma}^*) \mu^\xi(Z)\rangle=\langle(\mathrm{id}_{T^*M}-\gamma^*) \mu^\xi(Z)\rangle + \langle( \mu^\xi(\Xi(Z))\rangle.
\end{equation*}
Since $\Xi(Z)=P^\sharp dQ(Z)$ and the action is canonical, we have $ \langle( \mu^\xi(\Xi(Z))\rangle=P^\sharp(\langle dQ(Z)\rangle) = 0 $. Hence,
\begin{equation*}
(\mathrm{id}_{T^*M}-\bar{\gamma}^*) \mu^\xi(Z)=\langle(\mathrm{id}_{T^*M}-\gamma^*) \mu^\xi(Z)\rangle=\langle(\mathrm{id}_{T^*M}-\gamma^*) \mu^\xi\rangle(Z),
\end{equation*}
for all $Z\in \Gamma_{\operatorname{pr}}(\mathbb{H}^{\bar{\gamma}})$.
\end{proof}

It follows from Lemma \ref{adiabaticprop} that if the pre-momentum map satisfies the adiabatic condition \eqref{adiacond} relative to $\gamma$ then it takes value in the $\bar{\gamma}$-vertical 1-forms, i.e., $\mu^\xi\in \Gamma(\mathbb{H}^{\bar{\gamma}})^0 $. Moreover, one can say that the Hannay-Berry connection of $\gamma$ satisfies the adiabatic condition if the condition \eqref{adiaprop} holds, or if the pre-momentum map takes values in the space of $\bar{\gamma}$-vertical 1-forms.

Now, we arrive at  the following generalized version of the axiomatic definition \cite{MaMoRa-90} of Hannay-Berry type connections satisfying the adiabatic condition.

\begin{theorem}\label{axionhann}
Given an Ehresmann-Poisson connection $\gamma$  on $(M,\mathcal{F},P)$, suppose that  there exists another Ehresmann connection
$\tilde{\gamma}$ on the Poisson foliation which satisfies the following
conditions for all $X\in\Gamma_{\operatorname*{pr}}(TM)$:%
\begin{equation}
\mathbf{i}_{(\operatorname{id}-\tilde{\gamma})(X)}\mu=0,\label{H1}%
\end{equation}
\begin{equation}
\Xi(X)=P^{\sharp}dQ(X),\label{H2}%
\end{equation}
where $\Xi=\gamma-\tilde{\gamma}$ and $Q\in\Gamma(\mathbb{V}^{0})$ is
horizontal 1-form such that
\begin{equation}
\langle Q(X)\rangle^{G}\in\operatorname{Casim}(M,P).\label{H3}%
\end{equation}
Then, $\tilde{\gamma}$. Furthermore, a connection $\tilde{\gamma}$ satisfying \eqref{H1}-\eqref{H3} exists if and only if
\begin{equation}
\langle\mathbf{i}_{(\operatorname{id}-\gamma)(X)}\mu\rangle^{G}=0\text{ \ }\forall
X\in\Gamma_{\operatorname*{pr}}(TM)\label{H4}%
\end{equation}
\end{theorem}
\begin{proof} (\textit{Uniqueness}). Suppose we have two connections
$\tilde{\gamma}_{1}$ and  $\tilde{\gamma}_{2}$ satisfying \eqref{H1}-\eqref{H3}. By condition \eqref{H2}, we have
\begin{equation}
(\gamma_1-\gamma_2)(X)=P^\sharp((Q_1-Q_2)(X)),
\end{equation}
for every $X\in\Gamma_{\mathrm{pr}}(TM)$. It follows form here and condition \eqref{H1} that the function $(Q_2-Q_1)(X)$ is $G$-invariant. Indeed,
$$\begin{array}{rcl}
L_{\xi_M}((Q_2-Q_1)(X))&=&-\mathrm{i}_{(\gamma_1-\gamma_2)(X)}\mu_a\vspace{2mm}\\
&=&-\mathrm{i}_{(\mathrm{id}-\gamma_2)(X)}\mu_a
+\mathrm{i}_{(\mathrm{id}-\gamma_1)(X)}\mu_a=0.
\end{array}$$
Thus, $(Q_2-Q_1)(X)=\langle(Q_2-Q_1)(X)\rangle^G=\langle Q_2(X)\rangle^G-\langle Q_1(X)\rangle^G$. So, by condition \eqref{H3}, $(Q_2-Q_1)(X)$ is a Casimir function which implies that  $(\gamma_1-\gamma_2)(X)=0$ for all $X\in\Gamma_{\operatorname*{pr}}(TM)$ and then  $\gamma_1=\gamma_2$.

(\textit{Existence}). First of all, for each pair of connection such that $\gamma = \Xi +\tilde{\gamma} $ we have the following identity
\begin{equation}\label{EQP2}
\langle\mathrm{i}_{(\mathrm{id}-\gamma)(X)}\mu_a \rangle^G
=\langle\mathrm{i}_{(\mathrm{id}-\tilde{\gamma})(X)}\mu_a\rangle^G
-\langle\mathrm{i}_{\Xi(X)} \mu_a\rangle^G.
\end{equation}
Now, assume that there existe a connection $\tilde{\gamma}$ satisfying conditions \eqref{H1}-\eqref{H3}. Using \eqref{H2} and \eqref{H3}, we get that 
\begin{equation}\label{EQP3}
\begin{array}{rcl}
\langle\mathrm{i}_{\Xi(X)}\mu_a\rangle^G
&=&\langle\mathrm{i}_{P^{\sharp}d(Q(X))} \mu_a \rangle^G 
=-\langle\mathrm{i}_{P^{\sharp}\mu_a}d(Q(X)) \rangle^G
\vspace{2mm}\\
&=&-\mathrm{i}_{P^{\sharp}\mu_a}d(\langle Q(X)\rangle^G)
=\mathrm{i}_{P^{\sharp}d(\langle Q(X)\rangle^G)}\mu_a=0,
\end{array}
\end{equation}
for all $X\in\Gamma_{\mathrm{pr}}(TM)$. This relation together with identity \eqref{EQP2} and condition  \eqref{H1} 
imply \eqref{H4}. 
Conversely, suppose $\gamma$ satisfies \eqref{H4} and take $\tilde{\gamma}=\langle\gamma \rangle^G$. By Lemma \ref{adiabaticprop} $\tilde{\gamma}$ satisifies   satisfies condition \eqref{H1}. Also, the condition \eqref{H2} holds because of
the $G$-action admits a pre-momentum map. Finally,  the condition \eqref{H3} follows from the following identities
 $$0=\langle\langle X\rangle^G-X \rangle^G =\langle\Xi(X)\rangle^G=\langle P^{\sharp}d(Q(X))\rangle^G=
P^{\sharp}d\left(\langle Q(X)\rangle^G\right).$$
\end{proof}

\begin{corollary}
If a Poisson connection $\gamma$ satisfies the adiabatic condition \eqref{adiacond} then the Hannay-Berry connection  $\langle\gamma\rangle^G$ is the unique connection satisfiying the conditions \eqref{H1}-\eqref{H3}.
\end{corollary}

Given a Poisson connection $\gamma$ on the foliated Poisson manifold $(M,\mathcal{F},P)$, one can ask how to fix a pre-momentum map $\mu$ in order to satisfy the adiabatic condition \eqref{adiacond}. In particular, we wonder if there are some cohomological obstructions to the existence of such a $\mu$ .By Proposition \ref{casimirinhor} and Lemma \ref{adiabaticprop}, it follows that $\langle (\mathrm{id}_{T*M}-\gamma^*)(\mu^\xi)\rangle \in \mathcal{C}^1$ for all $\xi\in \mathfrak{g}$ but is not a cocycle of $d^\gamma_{1,0}$ in general. But, when it does, we can formulate an adiabaticity criterion for the existence of a momentum map satisfying the adiabatic conditions in terms of the de Rham-Casimir complex

\begin{proposition}\label{adiabcrit}
Assume that 
\begin{equation}
d^\gamma_{1,0}\langle (\mathrm{id}_{T*M}-\gamma^*)(\mu^\xi)\rangle =0.\label{coclycondition}
\end{equation}
 Then, there exists a pre-momentum map satisfying the adiabatic condition \eqref{adiacond} relative to $\gamma$ if and only the cohomology class of $\langle (\mathrm{id}_{T*M}-\gamma^*)(\mu^\xi)\rangle$ in the de Rham-Casimir complex is trivial.
\end{proposition} 
\begin{proof}
First, assume the class of $\langle (\mathrm{id}_{T*M}-\gamma^*)(\mu^\xi)\rangle $ is trivial, that is, for every $\xi \in \mathfrak{g} $ there exists a Casimir function $K_\xi\in C^\infty(M)$ such that
\begin{equation*}
\langle (\mathrm{id}_{T*M}-\gamma^*)(\mu^\xi)\rangle  = d^\gamma_{1,0} K^\xi.
\end{equation*}
Now, we define $\tilde{\mu}:\mathfrak{g}\rightarrow \Omega^1(M)$ by $\tilde{\mu}\xi= \mu^\xi -dK^\xi.$  It can be easily prove that $\tilde{\mu}$ is a pre-momentum for the $G$-action. Next,  $\tilde{\mu}$ satisfies the adiabatic condition. Indeed, for every $Z\in \Gamma_{\operatorname{pr}}(TM)$, we have
\begin{eqnarray*}
(\mathrm{id}-\bar{\gamma}^*)\tilde{\mu}\xi&=& (\mathrm{id}-\bar{\gamma}^*)\mu\xi(Z) - d^{\bar{\gamma}}_{1,0}K^\xi(Z),\\
&=& \langle \mathrm{id}-\gamma^*)\mu\xi\rangle (Z) - d^{\bar{\gamma}}_{1,0}K^\xi(Z),\\
&=& L_{(\gamma - \bar{\gamma})(Z)} K^\xi= \{ Q(Z), K^\xi\}_P=0.
\end{eqnarray*} 
Conversely, if a pre-momentum $\mu$ satisfies the adiabatic condition then the cohomology class of $\langle (\mathrm{id}_{T*M}-\gamma^*)(\mu^\xi)\rangle$ is trivial.
\end{proof}

Since $ \langle (\mathrm{id}_{T*M}-\gamma^*)(\mu^\xi)\rangle $ is a vertical form, the assumption in Proposition \ref{adiabcrit} means that $\langle (\mathrm{id}_{T*M}-\gamma^*)(\mu^\xi)\rangle $ is a cocycle of the operator $d^\gamma_{1,0}$.

\begin{corollary} If the pre-momentum map $\mu$ is locally Hamiltonian ($d\mu^\xi=0$), then the assumption of Proposition \ref{adiabcrit} always holds. 
\end{corollary}

In particularly, in the case of a canonical $G$-action with momentum map $\mathbb{J}:\mathfrak{g}\rightarrow C^\infty(M)$, the 1-cocycle in \eqref{coclycondition} is describe as follows. Its cohomology class is trivial in the following situations: 
\begin{itemize}

\item[(a)] For every $\xi\in \mathfrak{g}$ there exists Casimir function  $K^\xi\in C^\infty(M)$ such that
\begin{equation*}
\langle d^{\gamma}_{1,0}  \mathbb{J}^\xi\rangle^G = d^{\bar{\gamma}}_{1,0} K^\xi.
\end{equation*}

\item[(b)] The momentum map is equivariant and the Lie group is semisimple.

\end{itemize}
\section{Applications}

Assume again that we start with a Poisson foliation $(M,\mathcal{F},P)$ equipped with an action $\Phi:G\times M\rightarrow M$ which admits a pre-momentum map $\mu$, where $G$ is a connected and compact Lie group. In other words, we assume that the $G$-action satisfies conditions (\ref{C1}), \eqref{C2}. Our point is to construct $G$-invariant Dirac structures on
$(M,\mathcal{F},P)$ by combining the averaging procedure for Poisson connections in 
$\operatorname{Con}_{\operatorname{H}}(M,\mathcal{F},P)$ with the so-called coupling method (see also \cite{Cu-90,DW-08,Va-06,VaVo-14}).

First, recall some facts from the theory of Dirac structures. A subbundle $D\subset\mathbb{T}M=TM\oplus T^{\ast}M$ is said to be a Dirac structure if $D$ is maximally isotropic with respect to the natural pairing
\[
\langle (X,\alpha),(Y,\beta)\rangle =\beta(X)+\alpha(Y)%
\]
and involutive with respect to the Courant bracket%
\[
\lbrack(X,\alpha),(Y,\beta)]=([X,Y],L_{X}\beta-L_{Y}\alpha+\frac{1}{2}%
d(\alpha(Y)-\beta(X)).
\]
Every Dirac structure $D$ induces a pre-symplectic (singular) foliation
$(\mathcal{S},\omega)$ on $M$, where%
\[
T\mathcal{S}=\operatorname{pr}_{TM}(D)
\]
($\operatorname*{pr}_{TM}:TM\oplus T^{\ast}M\rightarrow TM$ is a natural
projection onto the first factor) and $\omega$ is a (smooth) leafwise presymplectic form defined at
each point $m\in M$ by
\[
\omega_{m}(X,Y)=\alpha(Y)
\]
for $\alpha\in T_{m}^{\ast}M$ such that $(X,\alpha)\in D_{m}$. On the
contrary, each pre-symplectic foliation $(\mathcal{S},\omega)$ on $M$, induces
a Dirac structure%
\[
D_{m}:=\{(X,\alpha)\mid X\in T_{m}\mathcal{S},\text{ }\alpha\mid
_{T_{m}\mathcal{S}}=-\mathbf{i}_{X}\omega\}.
\]
Now, pick a $\gamma\in\operatorname{Con}_{\operatorname{H}}(M,\mathcal{F},P)$
and fix a 2-form $\sigma=\sigma^{\gamma}$ in (\ref{HAM1}). Then, 
one can introduce the following distribution
$D^{\gamma,\sigma}\subset TM\oplus T^{\ast}M$ given by
\begin{equation}
D^{\gamma,\sigma}:=\{(X+P^{\sharp}(\alpha),\alpha-\mathbf{i}_{X}\sigma)\mid
X\in\Gamma\mathbb{H}^{\gamma}, \alpha\in\Gamma(\mathbb{H}^{\gamma}%
)^{0}\}.\label{AD}%
\end{equation}
It is clear that $D^{\gamma,\sigma}$ is a regular distribution whose rank is
just equal to $\dim M$. By straightforward computation, one can show that $D^{\gamma,\sigma}$ is a Lagrangian distribution.

\begin{proposition}\label{diracadm}
For every admissible Poisson connection $\gamma\in\operatorname{Con}%
_{\operatorname{H}}(M,\mathcal{F},P)$, the associated distribution
$D^{\gamma,\sigma}$ in (\ref{AD}) is a Dirac structure on $M$.  
\end{proposition}
\begin{proof}
We only need to prove that $D^{\gamma,\sigma}$ is closed under the Courant bracket. 
Taking into account that $$D^{\gamma,\sigma}=\mathrm{Graph}(P)\oplus 
\mathrm{Graph}(\sigma),$$ 
we fix the set (local) of generators of $D$ defined by the elements of the form
$$e_{\alpha}=(P^{\sharp}(\alpha),\alpha)\ \ \ \text{and}\ \ \ 
e_{X}=(X,-\mathbf{i}_X\sigma),$$
with $X\in\Gamma(\mathbb{H}^{\gamma})$ and 
$\alpha\in\Gamma((\mathbb{H}^{\gamma})^0)$. 
Since $\gamma$ is a Poisson connection we have 
$$[e_{\alpha},e_{\beta}]=\left(P^{\sharp}\left(L_{P^{\sharp}(\alpha)}\beta -\mathbf{i}_{P^{\sharp}(\beta)}d\alpha\right),L_{P^{\sharp}(\alpha)}\beta -\mathbf{i}_{P^{\sharp}(\beta)}d\alpha\right)\in D^{\gamma,\sigma},$$
and
$$[e_X,e_{\alpha}]=\left(P^{\sharp}(L_X\alpha),L_X\alpha+\mathbf{i}_{P^{\sharp}(\alpha)}d\mathbf{i}_X\sigma \right)\in D^{\gamma,\sigma}.$$
The admissibility of $\gamma$ implies that
$$[e_{X},e_Y]=\left([X,Y],-L_X\mathbf{i}_Y\sigma +\mathbf{i}_Yd\mathbf{i}_X\sigma \right)=\left([X,Y],-\mathbf{i}_{[X,Y]}\sigma\right)\in D^{\gamma,\sigma}.$$
Finally, the equations
$$\left\langle [e_X,e_{\alpha}], e_{\beta}\right\rangle = 0,\ \ \ \left\langle [e_X,e_{\alpha}], e_X\right\rangle = 0$$
hold because of $\gamma \in\operatorname{Con}%
_{\operatorname{H}}(M,\mathcal{F},P).$
\end{proof}

\begin{remark}
In fact, $D^{\gamma,\sigma}$ is  a coupling Dirac structure on the foliated
manifold $(M,\mathcal{F)}$ associated to the geometric data $(\gamma
,\sigma,P)$, \cite{Va-06,VaVo-14}.
\end{remark}
Recall that a distribution $D \subset TM\oplus T^*M$ is said to be $G$-invariant if 
\begin{equation}
(X,\alpha)\in\Gamma(D)\Longrightarrow(\Phi_{g}^{\ast}X,\Phi_{g}^{\ast}%
\alpha)\in\Gamma(D) \ \forall g\in G.\label{distinv}
\end{equation}
In particular, if a Dirac structure $D\subset TM\oplus T^*M$ is $G$-invariant as above, we will call the action a \textbf{Dirac action} of $D$

\begin{lemma}\label{Lemma1}
Let $\gamma\in\operatorname{Con}_{\operatorname{H}}(M,\mathcal{F},P)$ be an
arbitrary connection and $\bar{\gamma}=\langle\gamma\rangle^{G}$ its  $G$-average. Then,
the invariance of the distribution $D^{\bar{\gamma},\bar{\sigma}}$ under the
$G$-action is equivalent to the $G$-invariance of the 2-form $\bar{\sigma}$
in (\ref{MF1}),
\[
L_{\xi_{M}}\bar{\sigma}=0\ \forall a\in\mathfrak{g}.
\]
\end{lemma}
\begin{proof}
The invariance property for the averaged connection $\bar{\gamma}$ implies
that the corresponding splittings (\ref{SP1}) and (\ref{SP2}) are also
invariant under the $G$-action. The $G$-invariance condition for
$D^{\bar{\gamma},\bar{\sigma}}$ means that for any sections $X\in
\Gamma\mathbb{H}^{\gamma},$ $\alpha\in\Gamma(\mathbb{H}^{\gamma})^{0}$ and
$g\in G$, we have
\begin{equation}
\Phi_{g}^{\ast}X+\Phi_{g}^{\ast}P^{\sharp}(\alpha)=\tilde{X}+P^{\sharp}%
(\tilde{\alpha}),\label{I1}%
\end{equation}%
\begin{equation}
\Phi_{g}^{\ast}\alpha-\mathbf{i}_{\Phi_{g}^{\ast}X}\Phi_{g}^{\ast}%
\sigma=\tilde{\alpha}-\mathbf{i}_{\tilde{X}}\sigma\label{I2}%
\end{equation}
for some $\tilde{X}\in\Gamma\mathbb{H}^{\bar{\gamma}}$ and $\tilde{\alpha}%
\in\Gamma(\mathbb{H}^{\bar{\gamma}})^{0}$. \ Taking into account that the
action preserves the vertical $\mathbb{V}$ and horizontal $\mathbb{H}%
^{\bar{\gamma}}$ distributions (co-distributions), \ from (\ref{I1}) we
conclude that $\Phi_{g}^{\ast}X\in\Gamma\mathbb{H}^{\bar{\gamma}},\Phi
_{g}^{\ast}P^{\sharp}(\alpha)\in\Gamma\mathbb{V}$ and hence $\tilde{X}%
=\Phi_{g}^{\ast}X$. \ Moreover, it follows from (\ref{I2}) that $\tilde
{\alpha}=\Phi_{g}^{\ast}\alpha$ and $\mathbf{i}_{\tilde{X}}\sigma
=\mathbf{i}_{\tilde{X}}\Phi_{g}^{\ast}\sigma$ for all $\tilde{X}\in
\Gamma\mathbb{H}^{\bar{\gamma}}$. This implies that $\Phi_{g}^{\ast}%
\sigma=\sigma$.
\end{proof}

Now, we formulated a generalized version of the averaging theorem for Dirac structures \cite{VaVo-14}.

\begin{theorem}\label{teodiracinv}
Let $\gamma\in\operatorname{Con}_{\operatorname{H}}(M,\mathcal{F},P)$ be an
admissible Poisson connection. Then,  the averaged Poisson connection
$\bar{\gamma}=\langle\gamma\rangle^{G}$is again admissible and induces a $G$-invariant
Dirac structure $D^{\bar{\gamma},\bar{\sigma}+C}$, where $C\in\mathcal{C}^{2}$
is an arbitrary $\bar{d}^{\gamma}$-cocycle,
\[
\bar{d}^{\gamma}C=0.
\]
Moreover, if the pre-momentum map satisfies the adiabatic condition, then the $G$-action is a Hamiltonian action for $D^{\bar{\gamma},\bar{\sigma}+C}$.
\end{theorem}
\begin{proof}
By Theorem \ref{theoavercon} and Proposition \ref{diracadm} the distribution $D^{\bar{\gamma},\bar{\sigma}+C}$ defines a Dirac structure. To prove the $G$-invariance of $D^{\bar{\gamma},\bar{\sigma}+C}$ let us consider presymplectic foliations $(\mathcal{S},\omega)$ and
$(\mathcal{S},\bar{\omega})$, associated to $D^{\gamma,\sigma}$ and
$D^{\bar{\gamma},\bar{\sigma}}$ respectively. 
The ccharacteristic distribution of $D^{\gamma,\sigma}$ is
$$TS=\mathbb{H}^{\gamma}\oplus P^{\sharp}(T^*M),$$
with presymplectic form $\omega_S=\sigma\oplus \tau_S$,
where $\tau$ is the leaf wise symplectic form of $P$.
On the other hand, the characteristic distribution of 
$D^{\bar{\gamma},\bar{\sigma}}$ is
$$TS=\mathbb{H}^{\bar{\gamma}}\oplus P^{\sharp}(T^*M),$$
with presymplectic form $\bar{\omega}_S=\bar{\sigma}\oplus\tau_S$.
A generating family of vector fields for $TS$ is 
$$\left\{\widetilde{X}=X+P^{\sharp}d(Q(X)) \ \ \text{and}
\ \ P^{\sharp}(d f)\ |\ X\in\Gamma(\mathbb{H}^{\gamma}),\ 
f\in C^{\infty}(M)\right\}.$$
Evaluating $\bar{\omega}_S$ on the generating elements, we conclude that
\begin{equation*}
\bar{\omega}_S=\omega_S-dQ|_S.
\end{equation*}
Since the $G$-action admits a pre-momentum map,
the average of $\omega_S$ can be written as $\langle \omega_S\rangle ^G=\omega_S-i^*_Sd Q$,
where $i_S\colon S\hookrightarrow M$ is the canonical injection, (see \cite{VaVo-14}). 
Hence, $\bar{\omega}_S=\left\langle \omega_S\right\rangle ^G$
and the $G$-invariance of $\bar{\sigma}$ follows from here. The $G$-invariance of $D^{\bar{\gamma},\bar{\sigma}+C}$ is a consequence of Lemma \ref{Lemma1}. 
\end{proof}

\begin{corollary} 
The Dirac structures $D^{\gamma,{\sigma}}$ and $D^{\bar{\gamma},\bar{\sigma}+C}$ are related by  gauge transformation defined by the horizontal 2-form $dQ+C$.
\end{corollary}

\begin{corollary}\label{hamcondition}
If the pre-momentum map satisfies the adiabatic condition \eqref{adiacond}, then the infinitesimal generators of the $G$-action are local generators for $D^{\bar{\gamma},\bar{\sigma}+C}$, with $C\in\mathcal{C}^{2}$, that is
\begin{equation}
(\xi_M, \mu^\xi) \in \Gamma(D^{\bar{\gamma},\bar{\sigma}+C}) \ \ \ \forall \xi\in\mathfrak{g}.\label{hamcond}
\end{equation}
\end{corollary}
\begin{proof}
By Lemma \ref{adiabaticprop} we have $\mu^\xi\in \Gamma(\mathbb{H}^{\bar{\gamma}})^0 $ for each $\xi\in\mathfrak{g}$. Hence, $(\xi_M, \mu^\xi) \in \Gamma(D^{\bar{\gamma},\bar{\sigma}+C})$  for all $\xi\in\mathfrak{g}, $
\end{proof}

In the case when the pre-momentum map is actually a momentum map, i.e. $\mu^\xi =d \mathbb{J}^\xi $ for some $\mathbb{J}^\xi\in C^\infty(M)$, the action is called Hamiltonian, \cite{BRAHIC2014901}. Indeed, if the Dirac structure is the graph of a Poisson tensor, then the action is Hamiltonian in the usual sense (the infinitesimal generators are Hamiltonians). By Corollary \ref{hamcondition}, the adiabatic condition \eqref{adiacond} implies that the $G$-action is  Hamiltonian on $D^{\bar{\gamma},\bar{\sigma}+C}$.




\section*{Acknowledgement}
The authors are very grateful to Eduardo Velasco-Barreras for fruitful discussions. The research was partially supported by CONACYT under the grants CB2013 no. 219631 and CB2015 no. 258302. I. Hasse thanks the supporting from CONACYT CB2015 no. 258302 as a postdoctoral fellow where some of this work was done.


\nocite{Va-06}
\nocite{WusLieGpSqExp}
\nocite{Vo-05}
\nocite{DW-08}
\nocite{Wa-08}
\nocite{CuWe-86}
\nocite{Cu-90}
\nocite{JRS-11}
\nocite{SeWe}

\bibliographystyle{acm}

\bibliography{bibliography.bib}
  


%

\end{document}